\documentclass[a4paper, 12pt]{amsart}
\usepackage{graphicx} 
\usepackage[T1]{fontenc}
\usepackage[utf8]{inputenc}
\usepackage{newpxtext,newpxmath}
\usepackage{csquotes,microtype,mathtools}
\usepackage[cm]{fullpage}
\usepackage[dvipsnames]{xcolor}
\definecolor{darkblue}{rgb}{0,0,0.7}
\definecolor{darkred}{rgb}{0.7,0,0}
\usepackage{enumitem}
\usepackage{amsmath,amsthm}
\usepackage{tikz,float}
\usepackage{caption}
\usepackage{subcaption}

\usepackage{comment}
\usepackage{float}
\usepackage[colorlinks=true, linkcolor=darkred, citecolor=darkblue, urlcolor=blue, pagebackref=true, breaklinks=true]{hyperref}
\usepackage[nameinlink]{cleveref}

\usetikzlibrary{positioning, shapes, fit, arrows}
\usepackage{tkz-graph}
\tikzstyle{vertex}=[circle, draw, inner sep=0pt, minimum size=6pt]

\newtheorem{theorem}{Theorem}[section]
\newtheorem{proposition}{Proposition}[section]
\newtheorem{lemma}{Lemma}[section]
\newtheorem{corollary}{Corollary}[section]

\theoremstyle{definition}
\newtheorem{definition}{Definition}[section]
\newtheorem{example}{Example}[section]

\theoremstyle{remark}
\newtheorem{remark}{Remark}[section]
\newtheorem{problem}{Problem}[section]
\newcommand{\lnkate}{\mathbin{\!\!/\!\!/\!}} %
\newcommand{\lk}{{\mathrm{lk}}}
\newcommand{\del}{{\mathrm{del}}}

\newcommand{\lfcg}{G^{\mathrm{lf}}_D}
\newcommand{\tcg}{G^{\mathbf{t}}_D}
\newcommand{\lfch}{\mathcal{H}^{\mathrm{lf}}_D}
\newcommand{\tch}{\mathcal{H}^{\mathbf{t}}_D}

\def\lk{\mathrm{lk}}
\def\del{\mathrm{del}}
\def\Ind{\mathrm{Ind}}
\def\Dlf{\mathrm{Dlf}}
\def\DT{\mathrm{DT}}

\def\H{\mathcal{H}}

\title{Vertex decomposable complexes of directed forests, conflict graphs and chordality}

\thanks{Last updated: \today}

\author{Priyavrat Deshpande}
\address{Chennai Mathematical Institute, India}
\email{pdeshpande@cmi.ac.in}

\author{Rutuja Sawant}
\address{Chennai Mathematical Institute, India}
\email{rutuja@cmi.ac.in}
\date{}
\thanks{The authors are partially supported by a grant from the Infosys Foundation}


\begin{document}

\begin{abstract}
Let $D$ be a multidigraph. We study the simplicial complex $\mathrm{Dlf}(D)$, whose vertices are the directed edges of $D$ and whose faces correspond to directed linear forests, that is, vertex-disjoint unions of directed paths. We also consider the related directed tree complex $\mathrm{DT}(D)$.

Our main approach is to associate with $D$ a simple graph encoding the local incompatibilities among the edges of $D$. Under mild acyclicity assumptions, we show that $\mathrm{Dlf}(D)$ and $\mathrm{DT}(D)$ can be realized as the independence complexes of respective graphs. 
This correspondence allows us to apply structural results from the theory of independence
complexes to obtain graph-theoretic criteria guaranteeing vertex decomposability, shellability, and sequential Cohen-Macaulayness of these complexes.

In particular, we describe explicit forbidden induced directed subgraphs that obstruct vertex decomposability, and we identify classes of multidigraphs---including certain acyclic multidigraphs and multidigraphs whose underlying graphs are forests or cycles—for which $\mathrm{Dlf}(D)$ and $\mathrm{DT}(D)$ are vertex decomposable. We also provide examples showing that these properties do not hold in general.
\end{abstract}

\maketitle

\section{Introduction}

Simplicial complexes arising from graphs and directed graphs have been a central theme in
combinatorial topology over the past two decades. Classical examples include independence
complexes, matching complexes, and complexes associated with forests or trees, which
provide a rich interaction between graph theory, topology, and combinatorial commutative
algebra. A particularly fruitful direction has been the study of complexes of directed trees
and forests, initiated by Kozlov \cite{kozlov99} and further developed by Engstr\"om \cite{Engstrom05}, Joji\'c\cite{Ducko12, Dusko13}, Marietti and Testa \cite{marietti08}, among others,
where homotopy, shellability, vertex decomposability, and Cohen-Macaulayness are investigated using structural properties of the underlying graphs.

In this paper, we study a closely related, but more restrictive object. Given a multidigraph
$D$, we consider the simplicial complex whose vertices are the directed edges of $D$ and
whose faces consist of those sets of edges that form a \emph{directed linear forest}, i.e.,
a vertex-disjoint union of directed paths. We call this complex the \emph{directed linear
forest complex} of $D$ and denote it by $\mathrm{Dlf}(D)$. A related complex, obtained by
allowing arbitrary directed forests, is the \emph{directed tree complex} $\mathrm{DT}(D)$.

Complexes defined in terms of collections of directed paths have appeared implicitly in
recent work of Caputi, Collari, Smith, and Di~Trani \cite{CaputiCollariSmith2026, CaputiCollariTraniSmith2022, CaputiCollariDiTrani2022, CaputiCollariDiTrani2024} under the name of \emph{multipath complexes},
arising from the study of path posets and multipath cohomology of directed graphs.
In that setting, the simplices correspond to a vertex-disjoint union of directed paths, exactly as in
$\mathrm{Dlf}(D)$. However, the emphasis there is on homological constructions, functoriality,
and algebraic structures associated with path posets. In contrast, our focus is purely
combinatorial-topological: we study $\mathrm{Dlf}(D)$ and $\mathrm{DT}(D)$ directly as simplicial
complexes and investigate when they satisfy strong combinatorial properties such as vertex
decomposability, shellability, and (sequential) Cohen-Macaulayness.

Our main idea is to translate these directed-forest complexes into the independence complexes of
ordinary graphs. To each multidigraph $D$, we associate a simple graph whose vertices
correspond to the edges of $D$, and where adjacency encodes precisely the local obstructions
to being part of a directed linear forest (or directed forest). Under mild acyclicity
assumptions on $D$, we prove that
\[
\mathrm{Dlf}(D) = \mathrm{Ind}(\lfcg)
\quad\text{and}\quad
\mathrm{DT}(D) = \mathrm{Ind}(\tcg).
\]
This reduction allows us to import powerful structural
results from the theory of independence complexes, in particular theorems of Woodroofe \cite{Woodroofe09} and others on vertex decomposable graphs and forbidden induced cycles.

Using this correspondence, we obtain graph-theoretic criteria on $D$ that guarantee vertex
decomposability of $\mathrm{Dlf}(D)$ and $\mathrm{DT}(D)$. These criteria are expressed in terms of explicit forbidden induced directed subgraphs, providing a transparent
combinatorial interpretation of topological properties. We also present families of
multidigraphs for which the directed linear forest complex fails to be vertex decomposable or even sequentially Cohen-Macaulay, showing that the favorable behavior does not hold in
general.

The contribution of this work is therefore twofold. First, it provides a direct and elementary combinatorial framework for studying complexes of directed linear forests and
directed trees, complementary to the path-poset and multipath cohomology approach. Second, it connects these complexes systematically with independence complexes, allowing known
results on chordal graphs and forbidden induced cycles to be translated into the directed
graph setting.

The paper is organized as follows. In \Cref{sec:prel} we recall basic notions on simplicial complexes and multidigraphs, and describe link and deletion operations in the directed linear forest and directed tree complexes. 
In \Cref{sec:DLF} we associate with a multidigraph $D$
a simple graph $\lfcg$ and show that $\mathrm{Dlf}(D)$ is an independence complex. 
We then derive vertex decomposability results via forbidden induced subgraphs. 
\Cref{sec:DT} treats the directed tree complex $\mathrm{DT}(D)$ using similar methods.
In \Cref{sec:hyperg}, we introduce conflict hypergraphs associated with a multidigraph, whose hyperedges encode minimal forbidden configurations, including directed cycles of arbitrary length. This allows us to realize the directed linear forest and directed tree complexes as independence complexes of hypergraphs for arbitrary multidigraphs, without acyclicity assumptions.
Finally, in \Cref{lastsection} we suggest open future research directions and open problems. 

\section{Preliminaries}\label{sec:prel}

In this section, we recall basic notions from simplicial complexes and multidigraph theory used throughout this paper. 
We first review simplicial complexes and their combinatorial properties, then introduce multidigraphs and their associated simplicial complexes: directed linear forest complexes and directed tree complexes.

\subsection{Simplicial complexes}

We first recall standard definitions related to simplicial complexes.

An (abstract) simplicial complex $\Delta$ on a finite vertex set $V$ is a collection of subsets of $V$ such that if $\tau \in \Delta$ and $\sigma \subseteq \tau$, then $\sigma \in \Delta$. The elements of $\Delta$ are its \emph{faces}, and the inclusion-maximal faces are its \emph{facets}. For a face $\sigma \in \Delta$, its \emph{dimension} is $\dim \sigma = |\sigma| - 1$, and the dimension of $\Delta$ is $\dim \Delta = \max \{\dim \sigma : \sigma \in \Delta \}$. A subcomplex $\Delta'$ of $\Delta$ is a simplicial complex such that if $\sigma$ is a face of $\Delta'$, then $\sigma$ is a face of $\Delta$. The \emph{empty complex} $\{\emptyset\}$ consists only of the empty set and has dimension $-1$. The \emph{void complex} $\emptyset$ has no faces.
Let $\sigma$ be a face of $\Delta$, then by $\Delta[\sigma]$ we denote the subcomplex induced by the vertices of $\sigma$ inside $\Delta$. 

A simplicial complex $\Delta$ is \emph{pure} if all its facets have the same dimension; otherwise it is \emph{non-pure}. For a nonnegative integer $d$, the \emph{$d$-skeleton} of $\Delta$ is the subcomplex of faces of dimension at most $d$. The \emph{pure $d$-skeleton} $\Delta^{[d]}$ is the subcomplex whose facets are exactly the $d$-dimensional faces of $\Delta$. A \emph{$d$-simplex} is a simplicial complex with a single facet of dimension $d$.

For simplicial complexes $\Delta_1$ and $\Delta_2$ on disjoint vertex sets
$V(\Delta_1)$ and $V(\Delta_2)$, their \emph{join} $\Delta_1 * \Delta_2$ has
 vertex set
$V(\Delta_1 * \Delta_2) = V(\Delta_1) \sqcup V(\Delta_2)$ and face set
\[
\Delta_1 * \Delta_2
= \{ \sigma_1 \sqcup \sigma_2 \mid \sigma_1 \in \Delta_1,\ \sigma_2 \in \Delta_2 \}.
\]

Two basic constructions for a simplicial complex $\Delta$ are the link and the deletion. For a face $\sigma$ of $\Delta$, the \emph{link} of $\sigma$ in $\Delta$ is
\[
\lk_{\Delta}(\sigma)
= \{ \tau \in \Delta \mid \tau \cap \sigma = \emptyset,\ \tau \cup \sigma \in \Delta \},
\]
and the \emph{deletion} of $\sigma$ in $\Delta$ is
\[
\del_{\Delta}(\sigma)
= \{ \tau \in \Delta \mid \tau \cap \sigma = \emptyset \}.
\]
If $\sigma = \{v\}$ is a vertex, we write $\lk_{\Delta}(v)$ and $\del_{\Delta}(v)$.

\medskip
We now recall some key combinatorial and topological properties of simplicial complexes.

\begin{definition}\label{def:VD}
A simplicial complex $\Delta$ is \emph{vertex decomposable} if either $\Delta$ is a
simplex (including $\emptyset$ and $\{\emptyset\}$), or there exists a vertex $v$ of $\Delta$ such that both
$\lk_{\Delta}(v)$ and $\del_{\Delta}(v)$ are vertex decomposable, and
every facet of $\del_{\Delta}(v)$ is also a facet of $\Delta$. A vertex $v$ for which every facet of $\del_{\Delta}(v)$ is a facet of $\Delta$ is called a \emph{shedding vertex}. 
\end{definition}

\noindent
We now recall shellable and Cohen-Macaulay simplicial complexes.

\begin{definition}\label{def:shellable}
A simplicial complex $\Delta$ is \emph{shellable} if its facets
$\sigma_1, \ldots, \sigma_s$ admits an ordering such that for each $2 \le i \le s$, the
subcomplex
\[
\langle \sigma_i \rangle \cap \langle \sigma_1, \ldots, \sigma_{i-1} \rangle
\]
is pure of dimension $\dim(\sigma_i)-1$. Such an ordering is a
\emph{shelling order}.
\end{definition}

It is a classical result—see \cite[Theorem 3.35]{JJ08} and \cite[Theorem 12.3]{kozlov2007}—that any shellable simplicial complex $\Delta$ is homotopy equivalent to a wedge of spheres. Moreover, the wedge contains a sphere of dimension $d$ if and only if $\Delta$ has facets of dimension $d$; no such facets, no $d$–sphere in the wedge.

\begin{definition}\label{def:CM}
A simplicial complex $\Delta$ is \emph{Cohen-Macaulay} over a field $\mathbb{K}$ if
\[
\widetilde{H}_i(\mathrm{lk}_{\Delta}(\sigma);\mathbb{K}) = 0
\]
for all $\sigma \in \Delta$ and all $i < \dim \mathrm{lk}_{\Delta}(\sigma)$.
\end{definition}

A non-pure extension of this notion is given by sequential Cohen-Macaulayness.

\begin{definition}\label{def:SCM}
A simplicial complex $\Delta$ is \emph{sequentially Cohen-Macaulay} over a field
$\mathbb{K}$ if for every $d \ge 0$ its pure $d$-skeleton $\Delta^{[d]}$ is Cohen-Macaulay over
$\mathbb{K}$.
\end{definition}

For simplicial complexes, the following implications hold
(see \cite[Theorem 3.33]{JJ08}):
\begin{equation}\label{eq:1}
\text{Vertex decomposable}
\;\Rightarrow\;
\text{Shellable}
\;\Rightarrow\;
\text{Sequentially Cohen-Macaulay}.
\end{equation}

\medskip
\subsection{Multidigraphs}

A \emph{multidigraph} $D$ consists of a pair of finite sets $(V(D),E(D))$, where $V(D)$ is the set of vertices and $E(D)$ is the set of edges, together with two functions
\[
s,t : E(D) \to V(D).
\]
For each edge $e \in E(D)$, the vertex $s(e)$ is called the \emph{source} of $e$ and $t(e)$ is called its \emph{target}. The edge $e$ is oriented from $s(e)$ to $t(e)$ and is written as $\overrightarrow{s(e)t(e)}$. A pair of vertices $\{u,v\} \subseteq V(D)$ is said to be adjacent if either $\overrightarrow{uv}\in E(D)$ or $\overrightarrow{vu}\in E(D)$.

\medskip
Two different edges $e,f \in E(D)$ are called \emph{parallel} if they have the same source and the same target, that is, if $s(e)=s(f)$ and $t(e)=t(f)$. They are said to form a \emph{$2$-cycle} if the source of each is the target of the other, i.e.\ if $s(e)=t(f)$ and $s(f)=t(e)$.

\medskip
A multidigraph $D$ is called a \emph{digraph} if it has no parallel
edges; that is, for any two distinct edges $e,f \in E(D)$, the ordered pairs satisfy
\[
(s(e),t(e)) \neq (s(f),t(f)).
\]
A digraph $D$ is called a \emph{simple digraph} if it has neither
parallel edges nor $2$-cycles; equivalently, for any two distinct
vertices $u \neq v$, at most one of $\overrightarrow{uv}$ and
$\overrightarrow{vu}$ lies in $E(D)$. A (simple, undirected) \emph{graph} can be regarded as a simple digraph whose edges are undirected.
\begin{definition}\label{def: cong digraphs}
Two multidigraphs $D_1=(V(D_1),E(D_1))$ and $D_2=(V(D_2),E(D_2))$ are said to be \emph{isomorphic} if there exist bijections
\[
\phi_V : V(D_1) \to V(D_2)
\quad \text{and} \quad
\phi_E : E(D_1) \to E(D_2)
\]
such that for every edge $e \in E(D_1)$ we have
\[
s(\phi_E(e))=\phi_V(s(e))
\quad\text{and}\quad
t(\phi_E(e))=\phi_V(t(e)).
\]

\noindent
In this situation, we write $D_1 \cong D_2$.
\end{definition}

\medskip
For a multidigraph $D$, its \emph{underlying graph}, denoted $D^{\mathrm{un}}$, is the simple undirected graph with vertex set $V(D)$ in which two vertices are adjacent whenever there is at least one edge of $D$ between them (see \Cref{fig: underlying graph}). A multidigraph $D$ is said to be connected if its underlying graph is connected; equivalently, for any two vertices $u,v \in V(D)$, there exists a path (ignoring edge directions) between $u$ and $v$.

A multidigraph $D'$ is called a subgraph of multidigraph $D$ if $V(D') \subseteq V(D)$ and $E(D') \subseteq E(D)$ with the source and target maps of $D'$ inherited from those of $D$. For a subset $U \subseteq V(D)$, the \emph{induced subgraph} of $D$ on $U$, denoted $D[U]$, is the multidigraph with
vertex set $U$ and edge set
\[
E(D[U]) = \{ e \in E(D) \mid s(e),t(e) \in U \},
\]
again with source and target inherited from $D$. More generally, for a subset $A \subseteq E(D)$, the \emph{subgraph of $D$ induced by $A$}, denoted $D[A]$, is the multidigraph with edge set $A$ and vertex set given by all endpoints of edges in $A$. 

For a vertex $v \in V(D)$, the \emph{in-degree} and \emph{out-degree} of $v$ in $D$ are given by
\[
d^{-}(v) = |\{ e \in E(D) : t(e)=v \}|, 
\qquad
d^{+}(v) = |\{ e \in E(D) : s(e)=v \}|.
\]

A \emph{directed cycle} is a connected digraph in which each vertex has in-degree and out-degree both equal to $1$. The length of a directed cycle is defined to be $|V(D)|$. A \emph{directed path} is obtained from a directed cycle by removing exactly one edge. A \emph{directed tree} is a connected digraph with no directed cycles in which no two distinct edges share the same target. A \emph{directed forest} is a digraph satisfying the same conditions as a directed tree, except that it is not required to be connected.

We now introduce several operations on multidigraphs that will be central to our description of links and deletions in the corresponding simplicial complexes.

\begin{definition}\label{lnkgph}
To a given multidigraph $D=(V(D), E(D))$ and an edge $e \in E(D)$, we associate two contraction-like subgraphs that we call \emph{link graphs}. 
First, let $D\lnkate e$ be the multidigraph obtained from $D$ by first deleting every edge whose source is $s(e)$, every edge whose target is $t(e)$, and those edge(s) whose unordered pair of endpoints is $\{s(e),t(e)\}$.  
Then identifying isolated vertices $s(e)$ and $t(e)$(see \Cref{diag:contraction1}) such that incidents remain unchanged.

Similarly, let $D_{\downarrow e}$ be the multidigraph formed from $D$ by first deleting all edges with target $t(e)$ or unordered pair of endpoints $\{s(e),t(e)\}$, and then identifying vertices $s(e)$ and $t(e)$ (see \Cref{diag:contraction2}).    
\end{definition}

Finally, let $D-e$ denote the multidigraph obtained from $D$ by removing the edge $e$ (see \Cref{diag:contraction3}).

\begin{figure}[H]
\centering
\begin{subfigure}{0.48\textwidth}
    \centering
    \begin{tikzpicture}
        \node (a) at (0,0) {$1$};
        \node (c) at (0,-2) {$2$};
        \node (b) at (2,0) {$4$};
        \node (d) at (2,-2) {$3$};

        \draw[->] (a)--(b) node[midway,above] {$e$};
        \draw[->] (-0.1,-0.25) -- (-0.1,-1.75);
        \draw[->] (b)--(c);
        \draw[->] (a)--(c);
        \draw[->] (b)--(d);
        \draw[->] (c)--(d);
        \draw[->] (2.1,-1.7) -- (2.1,-0.2);
    \end{tikzpicture}
    \caption{Multidigraph $D$}
    \label{diag:multidigraph}
\end{subfigure}
\hfill
\begin{subfigure}{0.48\textwidth}
    \centering
    \vspace{0.4cm}
    \begin{tikzpicture}
        \node (a) at (0,0) {$1$};
        \node (c) at (0,-2) {$2$};
        \node (b) at (2,0) {$4$};
        \node (d) at (2,-2) {$3$};

        \draw[-] (a)--(c);
        \draw[-] (a)--(b);
        \draw[-] (b)--(d);
        \draw[-] (c)--(d);
        \draw[-] (b)--(c);
    \end{tikzpicture}
    \caption{Underlying graph - $D^{\mathrm{un}}$}
    \label{diag:underlying-graph}
\end{subfigure}

\caption{}
\label{fig: underlying graph}
\end{figure}

\begin{figure}[H]
\centering
\begin{subfigure}{0.3\textwidth}
    \centering
    \begin{tikzpicture}
        \node (c) at (0,-2) {$2$};
        \node (b) at (2,0) {$1$};
        \node (d) at (2,-2) {$3$};

        \draw[->] (b)--(c);
        \draw[->] (b)--(d);
        \draw[->] (c)--(d);
    \end{tikzpicture}
    \caption{$D\lnkate e$}
    \label{diag:contraction1}
\end{subfigure}
\hfill
\begin{subfigure}{0.3\textwidth}
    \centering
    \vspace{0.4cm}
    \begin{tikzpicture}
        \node (c) at (0,-2) {$2$};
        \node (b) at (1,0) {$1$};
        \node (d) at (2,-2) {$3$};

        \draw[->] (b)--(c);
        \draw[->] (0.8,-0.2)--(-0.03,-1.8);
        \draw[->] (0.7,-0.2)--(-0.15,-1.8);
        \draw[->] (b)--(d);
        \draw[->] (c)--(d);
    \end{tikzpicture}
    \caption{$D_{\downarrow e}$}
    \label{diag:contraction2}
\end{subfigure}
\hfill
\begin{subfigure}{0.3\textwidth}
    \centering
    \begin{tikzpicture}
        \node (a) at (0,0) {$1$};
        \node (c) at (0,-2) {$2$};
        \node (b) at (2,0) {$4$};
        \node (d) at (2,-2) {$3$};

        \draw[->] (-0.1,-0.25) -- (-0.1,-1.75);
        \draw[->] (b)--(c);
        \draw[->] (a)--(c);
        \draw[->] (b)--(d);
        \draw[->] (c)--(d);
        \draw[->] (2.1,-1.7) -- (2.1,-0.2);
    \end{tikzpicture}
    \caption{$D-e$}
    \label{diag:contraction3}
\end{subfigure}

\caption{}
\label{fig:contraction}
\end{figure}

\medskip
\subsection{Associated Complexes}
We now define the simplicial complexes associated with multidigraphs, which are the central objects of this paper.

\begin{definition}\label{def: DLFC}
For a multidigraph $D$, the \emph{directed linear forests complex}, denoted by $\Dlf(D)$, is the simplicial complex on the vertex set $E(D)$ whose faces are the subsets $\sigma \subseteq E(D)$ such that the induced subgraph $D[\sigma]$ is a vertex-disjoint union of directed paths.
\end{definition}

The next lemma explains how taking the link or deletion of a vertex in $\Dlf(D)$ corresponds to operations on the multidigraph $D$.

\begin{lemma}\label{lem:lk and del for Dlf}
For the directed linear forests complex, for every $e \in E(D)$ we have:
\begin{enumerate}
    \item $\mathrm{lk}_{\Dlf(D)}(e) = \Dlf(D\lnkate e)$,
    \item $\mathrm{del}_{\Dlf(D)}(e) = \Dlf(D-e)$.
\end{enumerate}
\end{lemma}
\begin{proof}
\emph{(1)} 
Let $\sigma \in \lk_{\Dlf(D)}(e)$, then $\sigma \cup \{e\} \in \Dlf(D)$.
This is equivalent to requiring that no edge in $\sigma$ has the same source, the same target, or both endpoints in common with $e$. 
Because $e \notin \sigma$, this condition is precisely $\sigma \subseteq E(D\lnkate e)$. 
Moreover, identifying $s(e)$ and $t(e)$ to form a single vertex does not create any new incidences among the remaining edges of $D$. 
Hence, $(D\lnkate e)[\sigma]$ is a vertex-disjoint union of directed paths in $D\lnkate e$ if and only if $D[\sigma \cup \{e\}]$ is a vertex-disjoint union of directed paths in $D$. Therefore, $\lk_{\Dlf(D)}(e) = \Dlf(D\lnkate e)$.

\medskip
\noindent
\emph{(2)} Let $\sigma \in \del_{\Dlf(D)}(e)$. By definition, this means $e \notin \sigma$ and $D[\sigma]$ is a vertex-disjoint union of directed paths in $D$. This is equivalent to saying that $(D-e)[\sigma]$ is a vertex-disjoint union of directed paths in $D-e$. Therefore, we obtain $\del_{\Dlf(D)}(e) = \Dlf(D-e)$.
\end{proof}

\medskip
We now turn to a closely related simplicial complex naturally arising from a multidigraph.

\begin{definition}\label{def: DTC}
Given a multidigraph $D$, the \emph{directed tree complex}, denoted by $\DT(D)$, is the simplicial complex with vertex set $E(D)$ whose faces are those subsets $\sigma \subseteq E(D)$ for which the induced subgraph $D[\sigma]$ is a directed forest.
\end{definition}

\noindent
Note that any directed linear forest is, in particular, a directed forest. Hence, for every multidigraph $D$ we have the inclusion

\begin{equation}\label{eq:inclusion}
\Dlf(D) \subseteq \DT(D).
\end{equation}

\medskip
Analogously to \Cref{lem:lk and del for Dlf}, the link and deletion of a vertex $v$ in $\DT(D)$ can be interpreted as operations on the multidigraph $D$ itself.

\begin{lemma}\label{lem:lk and del in DT}
For the directed tree complex and for every $e \in E(D)$, the following hold:
\begin{enumerate}
    \item $\mathrm{lk}_{\DT(D)}(e) = \DT(D_{\downarrow e})$,
    \item $\mathrm{del}_{\DT(D)}(e) = \DT(D-e)$.
\end{enumerate}
\end{lemma}
\begin{proof}
\emph{(1)} Let $\sigma \in \lk_{\DT(D)}(e)$. By definition, this means that $\sigma \cup \{e\} \in \DT(D)$. This occurs exactly when no edge in $\sigma$ has the same target as $e$, and no edge in $\sigma$ has the same ordered pair of endpoints as $e$. Since $e \notin \sigma$, this condition is equivalent to requiring that $\sigma \subseteq E(D_{\downarrow e})$. Moreover, identifying $s(e)$ and $t(e)$ does not create any additional incidences among the remaining edges of $D$. Consequently, $D_{\downarrow_e}[\sigma]$ is a directed forest if and only if $D[\sigma \cup \{e\}]$ is a directed forest. Hence, $\lk_{\DT(D)}(e) = \DT(D_{\downarrow_e})$.

\medskip
\emph{(2)} Now let $\sigma \in \del_{\DT(D)}(e)$. By definition, this means that $e \notin \sigma$ and that $D[\sigma]$ is a directed forest. This is equivalent to the requirement that $(D - e)[\sigma]$ be a directed forest. Therefore, we obtain $\del_{\DT(D)}(e) = \DT(D - e)$.
\end{proof}
\medskip
Let $D_1=(V_1,E_1)$ and $D_2=(V_2,E_2)$ be multidigraphs with disjoint vertex sets.
Their \emph{disjoint union} $D_1 \sqcup D_2$ is the multidigraph given by
\[
V(D_1 \sqcup D_2) = V_1 \sqcup V_2, \qquad
E(D_1 \sqcup D_2) = E_1 \sqcup E_2,
\]
where the source and target maps on $E_1$ and $E_2$ agree with those of $D_1$ and $D_2$, respectively. In particular, we obtain:
\begin{itemize}
    \item $\Dlf(D_1 \sqcup D_2) = \Dlf(D_1) * \Dlf(D_2)$,
    \item $\DT(D_1 \sqcup D_2) = \DT(D_1) * \DT(D_2)$.
\end{itemize}

\medskip
We finish this subsection with a simple example.

\begin{lemma}\label{lem: directed cycle}
Let $D$ be a directed cycle on $n$ vertices. Then the complex of directed linear forests and the complex of directed trees agree. Furthermore, this complex is vertex decomposable.
\end{lemma}
\begin{proof}
Let $\sigma \in \DT(D)$. Assume $\sigma \notin \Dlf(D)$. Then $D[\sigma]$ contains a vertex whose out-degree is at least $2$. However, in a directed cycle every vertex has in-degree and out-degree equal to $1$, so such a vertex cannot occur. 
Hence $\DT(D) \subseteq \Dlf(D)$. 
As the reverse inclusion holds due to \Cref{eq:inclusion}, the two complexes are equal.

For any $\sigma \subset E(D)$ with $\sigma \neq E(D)$, the induced subgraph $D[\sigma]$ is a directed path with $|\sigma|$ edges. Thus $\sigma \in \Dlf(D)$. On the other hand, if $\sigma = E(D)$, then $D[\sigma] = D$, which is a directed cycle, and so $E(D) \notin \Dlf(D)$. It follows that the complex is the boundary of an $(n-1)$-simplex and therefore, by \cite[Proposition 2.2]{ProBill1980}, it is vertex decomposable.
\end{proof}

\subsection{Independence Complex}
Our main strategy is to represent directed linear forest complexes and directed tree complexes as independence complexes of appropriately constructed graphs. We therefore briefly review the required notions.

Let $G=(V(G),E(G))$ be a simple graph. A subset $S \subseteq V(G)$ is called \emph{independent} if no two vertices of $S$ are adjacent in $G$. The \emph{independence complex} of $G$, denoted $\Ind(G)$, is the simplicial complex whose vertex set is $V(G)$ and whose faces are exactly the independent sets of $G$. Independence complexes have been studied previously in \cite{EHRENBORG06}, \cite{ATCF05}, \cite{VANTUYL08}, \cite{Woodroofe09}.

For $n \ge 3$, the \emph{cycle} of length $n$, written $C_n$, is the graph with vertex set $V(C_n) = [n]$ and edge set
\[
E(C_n) = \bigl\{\{i,i+1\} : 1 \le i \le n-1\bigr\} \cup \bigl\{\{n,1\}\bigr\}.
\]
An \emph{induced cycle} in a graph $G$ is an induced subgraph of $G$ that is isomorphic to $C_n$ for some $n \ge 3$. A graph is called \emph{chordal} if it has no induced subgraph isomorphic to $C_n$ for any $n \ge 4$.

\medskip
The next theorem is a central tool in our proofs.

\begin{theorem}[Theorem~1, \cite{Woodroofe09}]\label{V.D of Ind}
If $G$ is a graph with no induced cycle of length other than $3$ or $5$, then $\Ind(G)$ is vertex decomposable (and hence shellable and
sequentially Cohen-Macaulay).
\end{theorem}

\section{The complex of Directed Linear Forests}\label{sec:DLF}

In this section, we study the directed linear forest complex $\Dlf(D)$ of a multidigraph $D$ by viewing it as the independence complex of a suitably constructed simple graph, allowing us to apply structural results on independence complexes.

\begin{definition}\label{def: CGDLFC}
Let $D$ be a multidigraph. We associate to $D$ a simple graph $G^{\mathrm{lf}}_D$, called the \emph{linear-forest conflict graph}, whose vertex set is $E(D)$, that is, $V(\lfcg)=E(D)$.
Two vertices $e,f \in E(D)$ are adjacent in the conflict graph if and only if at least one of the following holds:
\begin{enumerate}
    \item $s(e)=s(f)$,
    \item $t(e)=t(f)$,
    \item $\{s(e),t(e)\}=\{s(f),t(f)\}$.
\end{enumerate}    
\end{definition}

Consequently, adjacency in $\lfcg$ encodes exactly the local constraints that prohibit two edges of $D$ from appearing together in a directed linear forest.

However, these local constraints describe directed linear forests completely in terms of the independence complex of the linear-forest conflict graph only when the multidigraph has no long directed cycles.
Let us first define all the relevant terminologies. 

\begin{figure}[H]
\centering
\begin{subfigure}{0.48\textwidth}
\tikzset{vtx/.style={circle, fill=black, inner sep=1.5pt}}
    \centering
    \begin{tikzpicture}
    \node[vtx] (a) at (0,0) {};
    \node[vtx] (c) at (0,-2) {};
    \node[vtx] (b) at (2,0) {};
    \node[vtx] (d) at (2,-2) {};

        \draw[->] (0.2,0)--(1.8,0) node[midway,above] {$e_1$};
        \draw[->] (-0.1,-0.2) -- (-0.1,-1.8) node[midway,left] {$e_7$};
        \draw[->] (1.9,-0.1)--(0.1,-1.9) node[midway,above] {$e_2$};
        \draw[->] (0,-0.2)--(0,-1.8) node[midway,right] {$e_6$};
        \draw[->] (2,-0.2)--(2,-1.8) node[midway,left] {$e_3$};
        \draw[->] (0.2,-2)--(1.8,-2) node[midway,below] {$e_5$};
        \draw[->] (2.1,-1.8) -- (2.1,-0.2) node[midway,right] {$e_4$};
    \end{tikzpicture}
    \caption{Multidigraph $D$}
\end{subfigure}
\hfill
\begin{subfigure}{0.48\textwidth}
    \centering
    \vspace{0.4cm}
    \begin{tikzpicture}
        \node (a) at (0,0) {$e_1$};
        \node (c) at (-1,-1) {$e_6$};
        \node (b) at (2,0) {$e_4$};
        \node (d) at (1,-1) {$e_7$};
        \node (e) at (0,-2) {$e_2$};
        \node (f) at (2,-2) {$e_3$};
        \node (g) at (4,-2) {$e_5$};

        \draw (a)--(c);
        \draw (a)--(b);
        \draw (c)--(d);
        \draw (a)--(d);
        \draw (e)--(c);
        \draw (e)--(d);
        \draw (e)--(f);
        \draw (f)--(g);
        \draw (b)--(f);
    \end{tikzpicture}
    \caption{The associated linear-forest conflict graph $\lfcg$}
\end{subfigure}
\caption{}
\label{association of simple graph 1}
\end{figure}

\begin{definition}\label{def:acyclic}
A multidigraph is called \emph{acyclic} if it has no directed cycles.
\end{definition}

\begin{definition}\label{def: weakly acyclic}
A multidigraph $D$ is called \emph{$2$-acyclic} if it has no directed cycle of length at least $3$.
\end{definition}


\begin{remark}
A $2$-acyclic multidigraph may contain directed $2$-cycles, whereas an acyclic
multidigraph contains no directed cycles at all, and hence does not permit $2$-cycles.
\end{remark}

\begin{lemma}\label{Ind of Dlf}
Let $D$ be a $2$-acyclic multidigraph. Then $\Dlf(D) = \Ind(\lfcg).$
\end{lemma}
\begin{proof}
$(\subseteq)$ If $\sigma \in \mathrm{Dlf}(D)$, then $D[\sigma]$ is a vertex-disjoint union of
directed paths. Hence, no vertex of $D[\sigma]$ has in-degree or out-degree exceeding $1$,
and there is at most one directed edge between any unordered pair of vertices. Thus, no two
edges in $\sigma$ are adjacent in $\lfcg$, so $\sigma$ is independent.

$(\supseteq)$ If $\sigma$ is independent in $\lfcg$, then in $D[\sigma]$ no vertex has more than one incoming or outgoing edge, and no two edges share the same unordered pair of endpoints. 
Hence, each connected component of $D[\sigma]$ is either a directed path or a directed cycle. 
Since two oppositely oriented edges share the same unordered pair of endpoints, they are adjacent in $\lfcg$ and hence cannot both lie in an independent set.
Thus $\sigma \in \mathrm{Dlf}(D)$.
\end{proof}

\medskip
As a consequence, combinatorial properties of $\Dlf(D)$ can be studied via $\lfcg$. In particular, vertex decomposability of $\Dlf(D)$ is governed by the presence of
induced cycles in $\lfcg$.

We therefore examine when the conflict graph can be a cycle.
To describe the directed configurations corresponding to such cycles, we introduce the following notion.

\begin{definition}\label{def: ACL}
An \emph{alternating closed trail} in a multidigraph $D$ is a cyclic sequence of distinct edges
\[
e_0,e_1,\ldots,e_{k-1},e_0
\]
such that for every $i$ (indices taken modulo $k$):
\[
s(e_i)=s(e_{i+1}) \quad \text{or} \quad t(e_i)=t(e_{i+1}).
\]
\end{definition}
\begin{remark}\label{OBS 1}
Let $D$ be a multidigraph such that its associated conflict graph is isomorphic to $C_n$. Since the vertices of $\lfcg$ correspond to the edges of $D$, we  label
\[
V(C_n)=\{e_0,e_1,\ldots,e_{n-1}\}, \quad 
E(C_n)=\bigl\{\{e_i,e_{i+1}\} : 0 \le i \le n-2\bigr\} \cup \bigl\{\{e_{n-1},e_0\}\bigr\}.
\]

\begin{enumerate}
    \item  Let $n \ge 4$ and suppose that $D$ has a pair of parallel edges. Without loss of
generality, assume that $e_0$ is parallel to $e_1$, that is,
\[
t(e_0)=t(e_1) \quad \text{and} \quad s(e_0)=s(e_1).
\]
Since $e_2$ is adjacent to $e_1$ but not to $e_0$, we must have either
\[
t(e_1)=t(e_2) \ \text{and} \ t(e_0)\neq t(e_2),
\quad \text{or} \quad
s(e_1)=s(e_2) \ \text{and} \ s(e_0)\neq s(e_2).
\]
Both possibilities lead to contradictions. Hence, if a multidigraph $D$ has parallel edges, then the conflict graph is not isomorphic to $C_n$ with $n \ge 4$.

\item Up to isomorphism, the cycle $C_4$ corresponds to the four digraphs depicted in \Cref{fig:association_C4_onerow}. Among these, \Cref{fig:row1} is a digraph whose underlying graph is a tree, and \Cref{fig:row2} is a simple digraph that forms an alternating closed trail.

\begin{figure}[H]
\centering
\begin{subfigure}{0.24\textwidth}
\centering
\scalebox{0.85}{
\begin{tikzpicture}[>=stealth]
    \node (a) at (0,0) {$1$};
    \node (b) at (2,0) {$2$};
    \node (c) at (4,0) {$3$};

    \draw[->] (a) -- (b) node[midway,below] {$e_0$};
    \draw[->] (1.7,0.1) -- (0.2,0.1) node[midway,above] {$e_1$};
    \draw[->] (b) -- (c) node[midway,below] {$e_2$};
    \draw[->] (3.7,0.1) -- (2.2,0.1) node[midway,above] {$e_3$};
\end{tikzpicture}}
\caption{}\label{fig:row1}
\end{subfigure}
\hfill
\begin{subfigure}{0.24\textwidth}
\centering
\scalebox{0.85}{
\begin{tikzpicture}[>=stealth]
    \node (a) at (0,0) {$1$};
    \node (b) at (2,0) {$2$};
    \node (d) at (0,-2) {$4$};
    \node (c) at (2,-2) {$3$};

    \draw[->] (a) -- (b) node[midway,above] {$e_0$};
    \draw[->] (c) -- (b) node[midway,right] {$e_1$};
    \draw[->] (c) -- (d) node[midway,below] {$e_2$};
    \draw[->] (a) -- (d) node[midway,left] {$e_3$};
\end{tikzpicture}}
\caption{}\label{fig:row2}
\end{subfigure}
\hfill
\begin{subfigure}{0.24\textwidth}
\centering
\scalebox{0.85}{
\begin{tikzpicture}[>=stealth]
    \node (a) at (0,0) {$1$};
    \node (b) at (2,0) {$2$};
    \node (c) at (1,-2) {$3$};

    \draw[->] (a) -- (b) node[midway,below] {$e_0$};
    \draw[->] (1.7,0.1) -- (0.2,0.1) node[midway,above] {$e_1$};
    \draw[->] (c) -- (b) node[midway,right] {$e_3$};
    \draw[->] (c) -- (a) node[midway,left] {$e_2$};
\end{tikzpicture}}
\caption{}\label{fig:row3}
\end{subfigure}
\hfill
\begin{subfigure}{0.24\textwidth}
\centering
\scalebox{0.85}{
\begin{tikzpicture}[>=stealth]
    \node (a) at (0,0) {$1$};
    \node (b) at (2,0) {$2$};
    \node (c) at (1,-2) {$3$};

    \draw[->] (a) -- (b) node[midway,below] {$e_0$};
    \draw[->] (1.7,0.1) -- (0.2,0.1) node[midway,above] {$e_1$};
    \draw[->] (b) -- (c) node[midway,right] {$e_2$};
    \draw[->] (a) -- (c) node[midway,left] {$e_3$};
\end{tikzpicture}}
\caption{}\label{fig:row4}
\end{subfigure}
\caption{}
\label{fig:association_C4_onerow}
\end{figure}
\end{enumerate}
\end{remark}

\medskip
The following lemma describes the structure of digraphs whose linear-forest conflict graph is isomorphic to $C_n$ for $n \ge 5$.

\begin{lemma}\label{cycle_associated_graph}
Let $D$ be a digraph whose linear-forest conflict graph $\lfcg$ is isomorphic to $C_n$ with $n \ge 5$. Then the underlying graph of $D$ contains a cycle of length at least $3$. Moreover, if $D$ is simple, then $D$ must be an alternating closed trail, so such a simple digraph exists only when $n$ is even.
\end{lemma}
\begin{proof}
Suppose $\lfcg \cong C_n$ with $n \ge 5$. Since $e_i$ is adjacent to $e_{i+1}$ in $\lfcg$, each consecutive pair of edges shares either a common source, a common target, or the same unordered pair of endpoints. By \Cref{OBS 1} and the fact that $D$ is a digraph, $D$ has no parallel edges. Furthermore, $D$ cannot contain the subgraph in \Cref{fig:association_C4_onerow}, since otherwise \Cref{OBS 1} would imply that $\lfcg$ contains a $C_4$. Therefore, the underlying graph of $D$ contains a cycle of length at least $3$. 

Now assume $D$ is simple. Then parallel edges and $2$-cycles are forbidden, so adjacency in $\lfcg$ arises only from pairs of edges with a common source or a common target. Thus, along the cycle $e_0,e_1,\ldots,e_{n-1},e_0$, consecutive edges must alternately share a common source and a common target, forming an alternating closed trail. This is possible only when $n$ is even.
\end{proof}

We now combine the preceding remark with \Cref{Ind of Dlf} to obtain a sufficient criterion
for the vertex decomposability of the directed linear forest complex, formulated in terms of
forbidden induced subgraphs. 

\begin{theorem}\label{thm: VDDLFC}
Let $D$ be an acyclic multidigraph that contains no alternating closed trail. Then the directed linear forest complex $\Dlf(D)$ is vertex decomposable.
\end{theorem}
\begin{proof}
By \Cref{Ind of Dlf}, we have $\Dlf(D) = \Ind(\lfcg)$. By \Cref{OBS 1} and \Cref{cycle_associated_graph}, the graph $\lfcg$ has no induced subgraph isomorphic to $C_n$ for $n \ge 4$. Hence, $\lfcg$ is chordal. The result now follows from \Cref{V.D of Ind}.
\end{proof}

The above theorem shows that the absence of certain cyclic configurations guarantees strong topological properties. However, such behavior does not hold in general, as shown by the following example.

\begin{figure}[H]
\centering
\begin{tikzpicture}[>=stealth]

\node (1) at (0,0) {$1$};
\node (2) at (2,0) {$2$};
\node (3) at (4,0) {$3$};
\node at (5,0) {$\cdots$};
\node (4) at (6,0) {$n$};
\node (5) at (8,0) {$n+1$};

\draw[->] (1) -- (2) node[midway,below] {$e_1$};
\draw[->] (1.7,0.1) -- (0.2,0.1) node[midway,above] {$e_1'$};

\draw[->] (2) -- (3) node[midway,below] {$e_2$};
\draw[->] (3.7,0.1) -- (2.2,0.1) node[midway,above] {$e_2'$};

\draw[->] (4) -- (5) node[midway,below] {$e_n$};
\draw[->] (7.3,0.1) -- (6.2,0.1) node[midway,above] {$e_n'$};

\end{tikzpicture}
\caption{$L_n$}\label{fig: double directed string}
\end{figure}

\begin{example}\label{Ex 1}
Let $n \ge 2$. Consider the double directed string $L_n$ with $n+1$ vertices (see
\Cref{fig: double directed string}), given by
\[
V(L_n) = [n+1], \qquad
E(L_n) = \{ e_i = \overrightarrow{(i)(i+1)},\;
e_i' = \overrightarrow{(i+1)(i)} : i \in [n] \}.
\]
Observe that
\[
\{e_1, e_2, \ldots, e_{i-1}, e_i, e_{i+1}, \ldots, e_n\}
\quad \text{and} \quad
\{e_1', e_2', \ldots, e_{i-1}', e_i', e_{i+1}', \ldots, e_n'\}
\]
are precisely the two facets of dimension $(n-1)$ of $\Dlf(L_n)$. It follows that the pure
$(n-1)$-skeleton $(\Dlf(L_n))^{[n-1]}$ is disconnected. In particular,
\[
\widetilde{H}_0\!\left((\Dlf(L_n))^{[n-1]},\mathbb{Z}\right) \neq 0,
\]
so $\Dlf(L_n)$ fails to be sequentially Cohen–Macaulay. Therefore, by \Cref{eq:1}, $\Dlf(L_n)$ is not shellable, as previously established in \cite[Proposition 3.2]{CaputiCollariSmith2026}.
\end{example}

\begin{figure}[h]
\centering
\newdimen\R
\R=2.00cm

\newdimen\Rb
\Rb=1.75cm
\begin{tikzpicture}
\draw[xshift=5.0\R] (270:\R) node {$1$};
\draw[xshift=5.0\R] (225:\R) node {$2$};
\draw[xshift=5.0\R,fill] (180:\R) node {$3$};
\draw[xshift=5.0\R,fill] (135:\R) node {$4$};
\draw[xshift=5.0\R, fill] (90:\R) node {$5$};
\draw[xshift=5.0\R,fill] (45:\R) node {$6$};
\draw[xshift=5.0\R,fill] (0:\R) node {$7$};
\draw[xshift=5.0\R,fill] (315:\R) node {$n$};

\node[xshift=5.0\R] (v0) at (270:\R) { };
\node[xshift=5.0\R] (v1) at (225:\R) { };
\node[xshift=5.0\R] (v2) at (180:\R) { };
\node[xshift=5.0\R] (v3) at (135:\R) { };
\node[xshift=5.0\R] (v4) at (90:\R) { };
\node[xshift=5.0\R] (v5) at (45:\R) { };
\node[xshift=5.0\R] (v6) at (0:\R) { };
\node[xshift=5.0\R] (vn) at (315:\R) { };

\draw[-latex] (v0) to[bend left] (v1);
\draw[-latex] (v1) to[bend left] (v2);
\draw[-latex] (v2) to[bend left] (v3);
\draw[-latex] (v3) to[bend left] (v4);
\draw[-latex] (v4) to[bend left] (v5);
\draw[-latex] (v5) to[bend left] (v6);
\draw[-latex] (vn) to[bend left] (v0);

\draw[latex-] (v0) to[bend right] (v1);
\draw[latex-] (v1) to[bend right] (v2);
\draw[latex-] (v2) to[bend right] (v3);
\draw[latex-] (v3) to[bend right] (v4);
\draw[latex-] (v4) to[bend right] (v5);
\draw[latex-] (v5) to[bend right] (v6);
\draw[latex-] (vn) to[bend right] (v0);

\draw[xshift=5.0\R, fill] (292.5:\Rb) node[above left] {$e'_{n}$};
\draw[xshift=5.0\R,fill] (247.5:\Rb) node[above right] {$e'_1$};
\draw[xshift=5.0\R,fill] (210.5:\Rb)   node[above right] {$e'_2$};
\draw[xshift=5.0\R,fill] (150.5:\Rb)  node[below right] {$e'_3$};
\draw[xshift=5.0\R, fill] (112.5:\Rb)   node[below right] {$e'_4$};
\draw[xshift=5.0\R,fill] (67.5:\Rb) node[below left] {$e'_5$};
\draw[xshift=5.0\R,fill] (22.5:\Rb) node[below left] {$e'_6$};

\draw[xshift=5.0\R, fill] (292.5:\R) node[below right] {$e_{n}$};
\draw[xshift=5.0\R,fill] (247.5:\R) node[below left] {$e_1$};
\draw[xshift=5.0\R,fill] (202.5:\R)   node[below left] {$e_2$};
\draw[xshift=5.0\R,fill] (157.5:\R)  node[above left] {$e_3$};
\draw[xshift=5.0\R, fill] (112.5:\R)   node[above left] {$e_4$};
\draw[xshift=5.0\R,fill] (67.5:\R) node[above right] {$e_5$};
\draw[xshift=5.0\R,fill] (22.5:\R) node[above right] {$e_6$};
\draw[xshift=4.95\R,fill] (337.5:\R)  node {$\cdot$} ;
\draw[xshift=4.95\R,fill] (333:\R)  node {$\cdot$} ;
\draw[xshift=4.95\R,fill] (342:\R)  node {$\cdot$} ;
\end{tikzpicture}
\caption{$P_n$} 
\label{fig: double directed cycle}
\end{figure}

\begin{example}\label{Ex 2}
Let $n \ge 3$. Consider the double directed cycle $P_n$ on $n$ vertices (see \Cref{fig: double directed cycle}), defined by
\[
V(P_n) = \mathbb{Z}_{n}, \qquad
E(L_n) = \{ e_i = \overrightarrow{(i)(i+1)},\;
e_i' = \overrightarrow{(i+1)(i)} : i \in \mathbb{Z}_{n} \}.
\]
Observe that the facets of the pure $(n-2)$-skeleton $(\Dlf(P_n))^{[n-2]}$ are precisely the sets of the following types:
\[
\{e_i, e_{i+1}, \ldots, e_{i-2}\}
\quad \text{and} \quad
\{e_i', e_{i+1}', \ldots, e_{i-2}'\},
\]
for each $i \in \mathbb{Z}_{n}$, where all indices are taken modulo $n$.

Therefore, the pure $(n-2)$-skeleton $(\Dlf(P_n))^{[n-2]}$ is disconnected. In particular,
\[
\widetilde{H}_0\!\left((\Dlf(P_n))^{[n-2]},\mathbb{Z}\right) \neq 0,
\]
which implies that $\Dlf(P_n)$ is not sequentially Cohen-Macaulay. Consequently, by \Cref{eq:1}, $\Dlf(P_n)$ is not shellable, as established in \cite[Proposition 3.1]{CaputiCollariSmith2026}.
\end{example}

Finally, we examine the particular case where the underlying graph is a forest. In this context, the lack of cycles guarantees vertex decomposability, provided a mild condition on the induced subgraph is satisfied.

\begin{definition}\label{def:Ln free}
A multidigraph $D$ is called \emph{$L_{n}$-free} if it does not contain an induced subgraph isomorphic to $L_n$.
\end{definition}

\begin{theorem}\label{thm: VDL2F}
Let $D$ be an $L_2$-free multidigraph whose underlying graph is a forest. Then $\Dlf(D)$ is vertex decomposable.
\end{theorem}
\begin{proof}
By \Cref{Ind of Dlf}, we know that $\Dlf(D)=\Ind(\lfcg)$. Since $D$ is $L_2$-free by \Cref{OBS 1}, the graph $\lfcg$ does not contain an induced subgraph isomorphic to $C_4$. Moreover, because the underlying graph of $D$ is a forest by \Cref{cycle_associated_graph}, the graph $\lfcg$ has no induced subgraph isomorphic to $C_n$ for any $n \ge 5$. Therefore, $\lfcg$ is chordal, and the result follows from \Cref{V.D of Ind}. \end{proof}

\section{The Complex of Directed Trees}\label{sec:DT}

In this section, we relate the directed tree complex $\DT(D)$ of a multidigraph $D$ to the independence complex of an associated conflict graph.

The combinatorial properties of the directed tree complexes have already been studied.
For an acyclic simple digraph, Engstr\"om \cite[Corollary 2.10]{Engstrom05} showed that the
directed tree complex is shellable.
Later, Singh \cite[Theorem 5.1]{anurag21} proved that the directed tree complex is vertex decomposable for multidigraphs whose underlying graph is a forest.
In a related direction, Du\v{s}ko \cite{Dusko13} associated a simple graph to a digraph whose underlying graph is a tree and showed that the corresponding directed tree complex is vertex decomposable \cite[Theorem 15]{Dusko13}.
Motivated by these works, we introduce a unified construction that, among other things, generalizes the above results.

\begin{definition}\label{def: TCG}
Let $D$ be a multidigraph. We define the associated \emph{tree conflict graph} $\tcg$ whose vertex set is $E(D)$.
Two vertices $e,f \in E(D)$ are adjacent in $\tcg$ if and only if at least one of the following holds:
\begin{enumerate}
    \item $t(e)=t(f)$,
    \item $\{s(e),t(e)\}=\{s(f),t(f)\}$.
\end{enumerate}    
\end{definition}

Thus, adjacency in $\tcg$ records precisely the local obstructions preventing two edges of $D$ from simultaneously belonging to a directed forest. The construction is illustrated in \Cref{association of simple graph 2}.

\begin{figure}[H]
\centering
\tikzset{vtx/.style={circle, fill=black, inner sep=1.5pt}}
\begin{subfigure}{0.48\textwidth}
    \centering
    \begin{tikzpicture}
    \node[vtx] (a) at (0,0) {};
    \node[vtx] (c) at (0,-2) {};
    \node[vtx] (b) at (2,0) {};
    \node[vtx] (d) at (2,-2) {};

        \draw[->] (0.2,0)--(1.8,0) node[midway,above] {$e_1$};
        \draw[->] (-0.1,-0.2) -- (-0.1,-1.8) node[midway,left] {$e_7$};
        \draw[->] (1.9,-0.1)--(0.1,-1.9) node[midway,above] {$e_2$};
        \draw[->] (0,-0.2)--(0,-1.8) node[midway,right] {$e_6$};
        \draw[->] (2,-0.2)--(2,-1.8) node[midway,left] {$e_3$};
        \draw[->] (0.2,-2)--(1.8,-2) node[midway,below] {$e_5$};
        \draw[->] (2.1,-1.8) -- (2.1,-0.2) node[midway,right] {$e_4$};
    \end{tikzpicture}
    \caption{Multidigraph $D$}
\end{subfigure}
\hfill
\begin{subfigure}{0.48\textwidth}
    \centering
    \vspace{0.4cm}
    \begin{tikzpicture}
        \node (a) at (0,0) {$e_1$};
        \node (c) at (-1,-1) {$e_6$};
        \node (b) at (2,0) {$e_4$};
        \node (d) at (1,-1) {$e_7$};
        \node (e) at (0,-2) {$e_2$};
        \node (f) at (2,-2) {$e_3$};
        \node (g) at (4,-2) {$e_5$};

        \draw[-] (a)--(b);
        \draw[-] (c)--(d);
        \draw[-] (e)--(c);
        \draw[-] (e)--(d);
        \draw[-] (f)--(g);
        \draw[-] (b)--(f);
    \end{tikzpicture}
    \caption{The associated tree conflict graph $\tcg$}
\end{subfigure}
\caption{}
\label{association of simple graph 2}
\end{figure}

The following lemma shows that, under mild assumptions, the directed tree complex of a multidigraph can be interpreted exactly as the independence complex of its tree conflict graph.

\begin{lemma}\label{Ind of DT}
Let $D$ be a $2$-acyclic multidigraph. Then $\DT(D) = \Ind(\tcg).$
\end{lemma}
\begin{proof}
\noindent\emph{($\subseteq$)}
Let $\sigma \in \DT(D)$, then $D[\sigma]$ is a directed forest. Hence no vertex of $D[\sigma]$ has in-degree exceeding $1$, and there is at most one directed edge between any unordered pair of vertices. Thus no two edges in $\sigma$ are adjacent in $\tcg$, so $\sigma$ is independent. 

\noindent\emph{($\supseteq$)}
If $\sigma$ is independent in $\tcg$, then in $D[\sigma]$ no vertex has more than one incoming edge, and no two edges share the same unordered pair of endpoints. Hence, each connected component of $D[\sigma]$ is either a directed tree or contains a directed cycle as an induced subgraph. Since $D$ contains no directed cycle of length at least $3$ and $2-$ cycles are excluded by the construction of $\tcg$, $D[\sigma]$ contains no directed cycle. Thus, $\sigma \in \DT(D)$.
\end{proof}

We now investigate the structural property of the tree conflict graph $\tcg$. In particular, we show that $\tcg$ is chordal under mild assumptions on the multidigraph $D$.

\begin{theorem}\label{chordal graph correspondence}
If $D$ is a $2$-acyclic multidigraph, then $\tcg$ is a chordal graph.
\end{theorem}
\begin{proof}
Assume, aiming for a contradiction, that $G^{\mathrm{t}}_D$ contains an induced cycle $C_n$ with $n \ge 4$,
\[
V(C_n)=\{e_1,\ldots,e_{n}\}, \quad 
E(C_n)=\bigl\{\{e_i,e_{i+1}\} : 1 \le i \le n-1\bigr\} \cup \bigl\{\{e_{n},e_1\}\bigr\}.
\]
Vertices of $G^{\mathrm{t}}_D$ represent edges of $D$. Since $C_n$ is an induced cycle, no nonconsecutive vertices on the cycle are adjacent in $G^{\mathrm{t}}_D$.
Recall that two edges of $D$ are adjacent in $G^{\mathrm{t}}_D$ precisely when either they share the same target (we will refer to such adjacencies as type $T$) or they have the same unordered pair of endpoints (these will be called type $E$).

\noindent
\textbf{Claim: The edges of $C_n$ must alternate in type.}

Suppose that both $\{e_i, e_{i+1}\}$ and $\{e_{i+1}, e_{i+2}\}$ are of type $T$.
Then, by transitivity, $t(e_i) = t(e_{i+1}) = t(e_{i+2})$, which implies that there is an edge $\{e_i, e_{i+2}\}$ in $G^{\mathrm{t}}_D$, contradicting the fact that $C_n$ is an induced cycle of length at least $4$.

An analogous argument shows that we cannot have two consecutive edges of type $E$ in $C_n$.
Therefore, the edge-types on $C_n$ must strictly alternate between $T$ and $E$.
In particular, the length of the cycle must be even, so write $n = 2k$.

\medskip
We now interpret the edges of $C_n$ back in the multidigraph $D$.
Assume that $\{e_1, e_2\}$ is of type $T$, so $e_1$ and $e_2$ have a common target $v_1$ in $D$.
Then $\{e_2, e_3\}$ must be of type $E$, meaning they share the same pair of endpoints.
Because $e_3$ and $e_1$ cannot also share the same target (otherwise $\{e_1,e_3\}$ would be an edge in $C_n$), it follows that $e_2 = \overrightarrow{v_2v_1}$ and $e_3 = \overrightarrow{v_1v_2}$ form a $2$-cycle in $D$.
By repeating this reasoning, $\{e_3, e_4\}$ is of type $T$ and $\{e_4, e_5\}$ is of type $E$, and so on around the cycle.
Thus, the induced cycle $C_n$ in $G^{\mathrm{t}}_D$ encodes a sequence of $2$-cycles in $D$ linked head-to-tail: $e_2, e_3$ between $v_1, v_2$; $e_4, e_5$ between $v_2, v_3$; etc.
These fit together into a directed walk given by the edges $e_1, e_3, e_5, e_7, \dots$ in $D$ that eventually closes to form a directed cycle.

\medskip

For each $i \in \{2,4,6,\ldots,n\}$, the pair $\{e_i,e_{i+1}\}$ determines a $2$-cycle (with the convention $e_{n+1}=e_1$), so following these edges yields a directed cycle of length $n \ge 3$ in $D$, contradicting the assumption that $D$ is $2$-acyclic.

\medskip
Hence $G^{\mathrm{t}}_D$ contains no induced cycle of length at least $4$, and therefore it is chordal.
\end{proof}

We now combine the preceding results to derive a key combinatorial consequence for the directed tree complexes.

\begin{corollary}\label{cor: V.D for DT(D)}
If $D$ is a $2$-acyclic multidigraph, then its directed forest complex $\DT(D)$ is vertex decomposable.
\end{corollary}
\begin{proof}
By \Cref{Ind of DT}, we have $\DT(D)=\Ind(\tcg)$.
By \Cref{chordal graph correspondence}, $\tcg$ is chordal.
The result now follows from \Cref{V.D of Ind}. 
\end{proof}

Note that the above method does not work for all multidigraphs, particularly when the underlying graph of $D$ has a directed cycle as a subgraph.
We therefore turn to a direct combinatorial analysis of the directed tree complex.
The next proposition illustrates that vertex decomposability can still be obtained for certain multidigraphs. 

\begin{proposition}\label{prop: V.D for DT(Cn)}
Let $D$ be a multidigraph whose underlying graph $D^{\mathrm{un}}$ is a cycle of length $n \ge 3$.
Suppose there exists a pair of adjacent vertices $\{u,v\} \subseteq V(D)$ such that either
$\overrightarrow{uv} \notin E(D)$ or $\overrightarrow{vu} \notin E(D)$.
Then $\DT(D)$ is vertex decomposable.
\end{proposition}

\begin{proof}
Fix $n \ge 3$. We prove the result by induction on the number of edges $m$ of $D$.
Since $D$ is a multidigraph whose underlying graph $D^{\mathrm{un}}$ is a cycle of length $n$,
we have $m \ge n$.

For $m = n$, the graph $D$ is a simple digraph, which is either cyclic or acyclic.
Hence, by \Cref{lem: directed cycle} and \Cref{cor: V.D for DT(D)},
$\DT(D)$ is vertex decomposable. Assume that $\Dlf(D')$ is vertex decomposable whenever $D'$ is a multidigraph satisfying the property and having strictly fewer than $m$ edges. 

Since both multidigraphs $D_{\downarrow e}$ and $D-e$ have strictly fewer edges than $D$, and each of them either satisfies the required property or has an underlying graph that is a tree, it follows from the induction hypothesis or from \Cref{cor: V.D for DT(D)} that $DT(D_{\downarrow e})$ and $\DT(D-e)$ are vertex decomposable.
Therefore, to show that $\DT(D)$ is vertex decomposable, it is enough to find a shedding vertex of $\DT(D)$.

Fix two parallel edges $e,f \in E(D)$, then $s(e)=s(f)$ and $t(e)=t(f)$.  
We claim that $e$ is a shedding vertex. 
If not, then there exists a facet $\sigma \in \del_{\DT(D)}(e)$ such that $\sigma \cup \{e\} \in \DT(D)$. However, since $e$ and $f$ have the same source and target, we also have
$\sigma \cup \{f\} \in \DT(D)$, which implies that
$\sigma \cup \{f\} \in \del_{\DT(D)}(e)$.
This contradicts the assumption that $\sigma$ is a facet of $\del_{\DT(D)}(e)$.

Therefore, we may assume that $D$ is a digraph satisfying the hypothesis.
If $D$ is simple, then the result follows.
Otherwise, suppose that $D$ contains a $2$-cycle between vertices $v_1$ and $v_2$; that is,
\[
\overrightarrow{v_1v_2}, \overrightarrow{v_2v_1} \in E(D).
\]
Since there exists a pair of adjacent vertices $\{u,v\} \subseteq V(D)$ such that either
$\overrightarrow{uv} \notin E(D)$ or $\overrightarrow{vu} \notin E(D)$, it suffices to verify the shedding vertex condition by considering the following four cases.

\begin{figure}[H]
\centering
\begin{subfigure}{0.24\textwidth}
\centering
\begin{tikzpicture}[scale=0.9]
    \node (a) at (0,0) {$v_1$};
    \node (c) at (0,-2) {$v_3$};
    \node (b) at (2,0) {$v_2$};
    \node (d) at (2,-2) {$v_4$};

    \draw[->] (0.3,0)--(1.7,0);
    \draw[->] (1.7,0.1) -- (0.3,0.1);
    \draw[->] (a)--(c);
    \draw[->] (b)--(d);
\end{tikzpicture}
\caption{Case I}\label{fig: Case I}
\end{subfigure}
\hfill
\begin{subfigure}{0.24\textwidth}
\centering
\begin{tikzpicture}[scale=0.9]
    \node (a) at (0,0) {$v_1$};
    \node (c) at (0,-2) {$v_3$};
    \node (b) at (2,0) {$v_2$};
    \node (d) at (2,-2) {$v_4$};

    \draw[->] (0.3,0)--(1.7,0);
    \draw[->] (1.7,0.1) -- (0.3,0.1);
    \draw[->] (c)--(a);
    \draw[->] (b)--(d);
\end{tikzpicture}
\caption{Case II}\label{fig: Case II}
\end{subfigure}
\hfill
\begin{subfigure}{0.24\textwidth}
\centering
\begin{tikzpicture}[scale=0.9]
    \node (a) at (0,0) {$v_1$};
    \node (c) at (0,-2) {$v_3$};
    \node (b) at (2,0) {$v_2$};
    \node (d) at (2,-2) {$v_4$};

    \draw[->] (0.3,0)--(1.7,0);
    \draw[->] (1.7,0.1) -- (0.3,0.1);
    \draw[->] (c)--(a);
    \draw[->] (2,-0.3)--(2,-1.7);
    \draw[->] (2.1,-1.7) -- (2.1,-0.3);
\end{tikzpicture}
\caption{Case III}\label{fig: Case III}
\end{subfigure}
\hfill
\begin{subfigure}{0.24\textwidth}
\centering
\begin{tikzpicture}[scale=0.9]
    \node (a) at (0,0) {$v_1$};
    \node (c) at (0,-2) {$v_3$};
    \node (b) at (2,0) {$v_2$};
    \node (d) at (2,-2) {$v_4$};

    \draw[->] (0.3,0)--(1.7,0);
    \draw[->] (1.7,0.1) -- (0.3,0.1);
    \draw[->] (a)--(c);
    \draw[->] (2,-0.3)--(2,-1.7);
    \draw[->] (2.1,-1.7) -- (2.1,-0.3);
\end{tikzpicture}
\caption{Case IV}\label{fig: Case IV}
\end{subfigure}
\caption{}
\label{fig:multidigraph-cases}
\end{figure}

\medskip
\noindent
\textbf{Case I:} 
$\overrightarrow{v_1v_3}, \overrightarrow{v_2v_4} \in E(D)$ and 
$\overrightarrow{v_3v_1}, \overrightarrow{v_4v_2} \notin E(D)$, 
or 
$\overrightarrow{v_3v_1}, \overrightarrow{v_4v_2} \in E(D)$ and 
$\overrightarrow{v_1v_3}, \overrightarrow{v_2v_4} \notin E(D)$.

In this case, $D$ is a multidigraph with no directed cycle of length at least $3$. Hence, by 
\Cref{cor: V.D for DT(D)}, the complex $\DT(D)$ is vertex decomposable.

\medskip
\noindent
\textbf{Case II:} 
$\overrightarrow{v_3v_1}, \overrightarrow{v_2v_4} \in E(D)$ and 
$\overrightarrow{v_1v_3}, \overrightarrow{v_4v_2} \notin E(D)$.

We claim that $\overrightarrow{v_2v_1}$ is a simplicial vertex in $\DT(D)$.  
Indeed, let $\sigma$ be a facet of 
$\del_{\DT(D)}(\overrightarrow{v_2v_1})$. Then either
$\overrightarrow{v_1v_2} \in \sigma$ or 
$\overrightarrow{v_3v_1} \in \sigma$. In both cases,
\[
\sigma \cup \{\overrightarrow{v_2v_1}\} \notin \DT(D).
\]
Therefore, $\overrightarrow{v_2v_1}$ is a simplicial vertex of $\DT(D)$.

\medskip
\noindent
\textbf{Case III:} 
$\overrightarrow{v_3v_1}, \overrightarrow{v_2v_4}, \overrightarrow{v_4v_2} \in E(D)$ and 
$\overrightarrow{v_1v_3} \notin E(D)$.

Using the same argument as in Case~II, we again conclude that
$\overrightarrow{v_2v_1}$ is a simplicial vertex in $\DT(D)$.

\medskip
\noindent
\textbf{Case IV:} 
$\overrightarrow{v_1v_3}, \overrightarrow{v_2v_4}, \overrightarrow{v_4v_2} \in E(D)$ and 
$\overrightarrow{v_3v_1} \notin E(D)$.

In this case, $\overrightarrow{v_1v_2}$ is a simplicial vertex of $\DT(D)$.  
Let $\sigma$ be a facet of 
$\del_{\DT(D)}(\overrightarrow{v_1v_2})$.

If $\overrightarrow{v_2v_1} \in \sigma$, then 
\[
\sigma \cup \{\overrightarrow{v_1v_2}\} \notin \DT(D).
\]

Now assume that $\overrightarrow{v_2v_1} \notin \sigma$.  
Since the in-degree of $v_1$ is equal to $1$, this can occur only if the underlying graph of $D[\sigma \cup \{\overrightarrow{v_2v_1}\}]$ is a cycle. Consequently, the underlying graph of
$D[\sigma \cup \{\overrightarrow{v_1v_2}\}]$ is also a cycle, and hence
\[
\sigma \cup \{\overrightarrow{v_1v_2}\} \notin \DT(D).
\]

Therefore, $\overrightarrow{v_1v_2}$ is a simplicial vertex in $\DT(D)$.

This completes the proof. 
\end{proof}

In \Cref{prop: V.D for DT(Cn)}, when $n = 3$ or $4$, the complex $\DT(D)$ is vertex decomposable for every multidigraph $D$ whose underlying graph $D^{\mathrm{un}}$ is an $n$-cycle.
For $n \ge 5$,\; $\DT(D)$ is vertex decomposable if and only if there is a pair of adjacent vertices $\{u,v\} \subseteq V(D)$ such that either $\overrightarrow{uv} \notin E(D)$ or $\overrightarrow{vu} \notin E(D)$.
This condition is necessary, since otherwise $\DT(D)$ has no shedding vertex.

It suffices to show that the complex associated with the double directed cycle $P_n$, defined in \Cref{Ex 2}, has no shedding vertex.
Consider $v = e_1$. The face \[ \sigma = \{e_2, e_3, \ldots, e_{n-3}, e_{n-2}', e_n\} \] is a facet of $\del_{\DT(P_n)}(e_1)$, but not a facet of $\DT(P_n)$, since $\sigma \cup \{e_1\} \in \DT(P_n)$. Hence, $e_1$ is not a shedding vertex of $\DT(P_n)$. We conclude that, by symmetry, $\DT(P_n)$ has no shedding vertex. 

\section{Going beyond pairwise conflicts}\label{sec:hyperg}

In the previous two sections, we associated to a multidigraph $D$ two simple graphs, the linear-forest conflict graph and the tree conflict graph, and showed that under the assumption of $2$-acyclicity the directed linear forest complex and the directed tree complex coincide with the independence complexes of these graphs.
This translation allowed us to apply structural results from the theory of independence complexes of graphs, in particular chordality criteria, to obtain vertex decomposability and shellability results.

The key feature underlying this approach is that, for $2$-acyclic multidigraphs, all minimal obstructions to being a directed linear forest or a directed tree are detected by pairwise incompatibilities between edges. 
Directed cycles of length at least three are excluded by assumption, while $2$-cycles and degree violations are ruled out by local adjacency conditions in the conflict graphs.

For general multidigraphs, however, this pairwise framework is no longer sufficient. 
Directed cycles of length at least three give rise to higher-order obstructions that cannot be encoded using edges of a simple graph alone. 
Consequently, the directed linear forest complex and the directed tree complex can no longer be realized as independence complexes of graphs once the $2$-acyclicity assumption is dropped.

In this section, we address this limitation by introducing conflict hypergraphs associated with a multidigraph $D$. 
The vertices of these hypergraphs are again the edges of $D$ but hyperedges now encode minimal forbidden configurations, including directed cycles of arbitrary length. 
We show that, for arbitrary multidigraphs, the directed linear forest complex and the directed tree complex can be realized as independence complexes of their respective conflict hypergraphs.

This hypergraph perspective allows us to go beyond acyclicity assumptions and provides a natural framework in which higher-order obstructions are treated uniformly. 
Moreover, by relating these conflict hypergraphs to the notion of $W$-chordality, we obtain shellability and sequential Cohen–Macaulayness results that extend the graph-based theory developed earlier.

We now recall the basic terminology concerning hypergraphs, their independence complexes, and the notion of $W$-chordality, which will serve as the main structural tool in extending shellability results beyond the $2$-acyclic case.

A \emph{hypergraph} $\H$ on a vertex set $V(\H)$ is a set of subsets
of $V(\H)$ called \emph{hyperedges} of $\H$ such that if $e_{1}$ and $e_{2}$ are distinct hyperedges of $\H$
then $e_{1}\not\subset e_{2}$. A $d$-hyperedge is a hyperedge consisting of exactly $d$ vertices, and a hypergraph is \emph{$d$-uniform} if every hyperedge has exactly $d$ vertices; otherwise it is a \emph{non-uniform} hypergraph

An \emph{independent set } of $\H$ is a subset of $V$ containing
no hyperedge. The \emph{independence complex} of $\H$, denoted $\Ind(\H)$, is the simplicial complex whose vertex set is $V(\H)$ and whose faces are exactly the independent sets of $\H$. 

Given a hypergraph $\H$, there are two ways of removing a vertex
that are of interest. Let $v\in V(\H)$. The \emph{deletion}
$\H \setminus v$ is the hypergraph on the vertex set $V(\H)\setminus\{v\}$
with hyperedges $\{e\,:\, e \in E(\H) \text{~and~} v\notin e\}$.
The \emph{contraction} $\H/v$ is the hypergraph on the vertex
set $V(\H)\setminus\{v\}$ with hyperedges the minimal sets
of $\{e\setminus\{v\}\,:\, e \in E(\H)\}$.
Thus, $\H\setminus v$ deletes all hyperedges containing $v$,
while $\H/v$ removes $v$ from each hyperedge containing it
and then removes any redundant hyperedges. A hypergraph $\H'$ obtained from $\H$ by repeated deletions and/or contractions is called a \emph{minor} of $\H$.

\begin{definition}\label{def: simplicial vertex of hypergraph} 
Let $\H$ be a hypergraph.
A vertex $v$ of $\H$ is \emph{simplicial} if for every
two circuits $e_{1}$ and $e_{2}$ of $\H$ that contain
$v$, there is a third circuit $e_{3}$ such that $e_{3}\subseteq(e_{1}\cup e_{2})\setminus\{v\}$.
\end{definition}

\begin{definition}\label{def: chordal hypergraph}
A hypergraph $\H$ is \emph{W-chordal} if every minor of $\H$
has a simplicial vertex. 
\end{definition}

\begin{theorem}[{\cite[Theorem 1.1]{RW11}}]\label{them: SCH}
If $\H$ is a W-chordal hypergraph, then the independence complex $\Ind(\H)$ is shellable and hence sequentially Cohen-Macaulay.
\end{theorem}

\subsection{Conflict hypergraph for the directed tree complex}

We begin with the directed tree complex. The associated conflict hypergraph encodes precisely those edge sets that prevent a collection
of directed edges from forming a directed tree.

\begin{definition}\label{def: conflict hypergraph for DT}
Let $D$ be a multidigraph. We associate to $D$ a hypergraph $\tch$, called the \emph{tree conflict hypergraph}, whose vertex set is $E(D)$, that is, $V(\tch)=E(D)$. Its hyperedge set is the following set:

\[
\begin{aligned}
E(\tch) = &\{\sigma \subseteq E(D) : D[\sigma] \text{ is a directed cycle}\} \;\cup\; \bigl\{\{e,f\} : e \text{ and } f \text{ are parallel edges}\bigr\}\\
& \;\cup\; \bigl\{\{e,f\} : s(e) \neq s(f) \text{~and~} t(e) = t(f)\bigr\}.
\end{aligned}
\]
\end{definition}

\begin{lemma}\label{Ind for H}
    Let $D$ be a multidigraph. Then $\DT(D) = \Ind(\tch)$
\end{lemma}
\begin{proof}
$(\subseteq)$
Let $\sigma \in \DT(D)$. Then $D[\sigma]$ is a directed tree.
Hence $D[\sigma]$ contains no directed cycle, no two parallel
edges, and no two edges with distinct sources and a common
target. Therefore $\sigma$ contains no hyperedge of $\tch$,
and so $\sigma$ is independent in $\tch$. Thus
$\sigma \in \Ind(\tch)$.

$(\supseteq)$
Let $\sigma$ be independent in $\tch$. Then $\sigma$ contains
no hyperedge of $\tch$. Hence $D[\sigma]$ has no directed
cycle, no parallel edges, and no pair of edges with distinct
sources and the same target. Consequently, every vertex has
in-degree at most $1$, and the digraph $D[\sigma]$ is acyclic. Therefore each connected component of $D[\sigma]$ is a directed tree, and hence
$\sigma \in \DT(D)$.
\end{proof}

We next describe how deletion and contraction in the tree
conflict hypergraph corresponds to natural operations on the
underlying multidigraph. These identities show that the
hypergraph operations required in the study of $W$-chordality
translate into deleting an edge or performing the link-type
reduction in the multidigraph. 

\begin{lemma}
Let $D$ be a multidigraph and $e\in E(D)$. For the tree conflict hypergraph, the following identities hold:
\begin{enumerate}
    \item $\tch \setminus \{e\} =  \H_{D-e}^{t}$
    \item $\tch / e = \H_{D_{\downarrow e}}^{t} \cup \big\{f : \{e,f\} \in E(\tch) \big\}$.
\end{enumerate}
\end{lemma}
\begin{proof}
Recall that $V(\tch)=E(D)$ and that hyperedges of $\tch$
correspond to directed cycles of $D$ together with conflicting
pairs of edges.

\medskip
\noindent\emph{(1)}
Since
\[
V(\tch\setminus\{e\})=E(D)\setminus\{e\}=E(D-e)
=V(\H_{D-e}^{t}),
\]
it suffices to compare hyperedges.
A set $\sigma\subseteq E(D)\setminus\{e\}$ is a hyperedge of
$\tch\setminus\{e\}$ if and only if it induces either a directed cycle or a
conflicting pair in $D$. As deleting $e$ does not affect configurations
avoiding $e$, the same condition characterizes hyperedges of
$\H_{D-e}^{t}$. Hence
\[
\tch\setminus\{e\}=\H_{D-e}^{t}.
\]

\medskip
\noindent\emph{(2)}
By the definition of hypergraph contraction,
hyperedges of $\tch/e$ are the minimal sets among
\[
\{\sigma\setminus\{e\}:\sigma\in E(\tch),\, e\in\sigma\}.
\]
If $\sigma$ induces a forbidden configuration containing $e$,
then after removing edges conflicting with $e$ and performing the identification described in the construction of the link
graph $D_{\downarrow e}$, the set $\sigma\setminus\{e\}$ induces
precisely a forbidden configuration in $D_{\downarrow e}$.
These yield the hyperedges of $\H_{D_{\downarrow e}}^{t}$.
If $\sigma=\{e,f\}$ is a conflicting pair, contraction produces the singleton $\{f\}$. These are exactly the hyperedges $\{f\}$ with $\{e,f\}\in E(\tch)$.
Therefore,
\[
\tch/e=\H_{D_{\downarrow e}}^{t}\cup
\bigl\{\{f\}:\{e,f\}\in E(\tch)\bigr\}.
\]
\end{proof}

\begin{example}
    Let $D = P_3$ be the multidigraph shown in \Cref{fig: double directed cycle}. For $e_1 \in E(P_3)$, we have $\{e_1,e_1'\}, \{e_1,e_2'\} \in E(\tch)$, whereas $D[\{e_1',e_2'\}]$ is a directed path and thus $\{e_1',e_2'\} \notin E(\tch)$. Consequently, $e_1$ is not a simplicial vertex of $\tch$. By symmetry, no edge of $P_3$ is a simplicial vertex of $\tch$. Note that $\DT(P_3)=\Ind(\tch)$ is vertex decomposable and therefore shellable, while $\tch$ itself is not $W$-chordal.
\end{example}

To relate the structural properties of the multidigraph $D$ to those of the hypergraph $\tch$, we next compare simplicial edges in $D$ with
simplicial vertices in the associated conflict hypergraph.
This correspondence will allow $W$-chordality of $\tch$ to be translated into chordality-type conditions on $D$.

\begin{definition}
Let $D$ be a multidigraph. An edge $e$ is said to be 
\emph{non-simplicial} in $D$ if there exists an edge 
$f \in E(D)$ with $t(e)=t(f)$ and $s(e) \neq s(f)$, and a 
directed cycle $D[\sigma]$ containing $e$, such that 
$s(f) \notin V(D[\sigma])$. 
Otherwise, $e$ is called a \emph{simplicial edge} of $D$.
\end{definition}

\begin{lemma}
Let $D$ be a multidigraph. Then an edge $e$ of $D$ is a simplicial
edge of $D$ if and only if $e$ is a simplicial vertex of $\tch$.
\end{lemma}

\begin{proof}
Recall that the vertices of $\tch$ are the edges of $D$, and its
hyperedges correspond to directed cycles together with pairs
$\{e,f\}$ where either $e$ and $f$ are parallel or
$t(e)=t(f)$ and $s(e)\neq s(f)$.

$(\Rightarrow)$
Suppose first that $e$ is not contained in any directed cycle.
Then the only hyperedges of $\tch$ containing $e$ are of the
form $\{e,f\}$ where $e$ and $f$ are parallel or
$t(e)=t(f)$ with $s(e)\neq s(f)$. If
$\{e,f\},\{e,f'\}\in E(\tch)$, then $f$ and $f'$ have a common
target, and hence $\{f,f'\}\in E(\tch)$. Thus $e$ is a
simplicial vertex of $\tch$.

Now suppose that $e$ belongs to directed cycles.
Let $D[\sigma]$ and $D[\sigma']$ be two directed cycles containing $e$, where $\sigma,\sigma' \subseteq E(D)$.
Then there exist edges
$f \in \sigma \setminus\{e\}$ and $f' \in \sigma' \setminus\{e\}$
entering the vertex $t(e)$, so that
$t(f)=t(f')$. Hence $\{f,f'\}$ is a hyperedge of $\tch$.

Moreover, if $D[\sigma]$ is a directed cycle containing $e$ and
there exists an edge $f\in E(D)$ with $t(f)=t(e)$, then by the definition of a simplicial edge, there exists a vertex
$v\in V(D[\sigma])$ such that $s(f)=v$. Then we obtain a directed cycle contained in $(\sigma \cup\{f\})\setminus\{e\}$ and containing $f$.
Hence, $e$ is a simplicial vertex of $\tch$.

$(\Leftarrow)$
Conversely, suppose that $e$ is a simplicial vertex of $\tch$.
Assume that $e$ belongs to a directed cycle $D[\sigma]$. If there
exists an edge $f$ with $t(f)=t(e)$ and
$s(f)\neq v$ for every $v\in V(D[\sigma])$, then both $\sigma$ and
$\{e,f\}$ are hyperedges of $\tch$, whereas
$(\sigma \cup\{f\})\setminus\{e\}$ induces a directed path and
contains no hyperedge, contradicting the simpliciality of $e$
in $\tch$. Hence, such an edge $f$ cannot exist, and therefore
$e$ is a simplicial edge of $D$.
\end{proof}

\begin{definition}
    Let $D$ be a multidigraph. An edge $e$ in $D$ is called a \emph{cycle-piercing edge} if there is a directed cycle $D[\sigma]$ in $D$ and an edge $f \in \sigma$ such that $t(e) = t(f)$ but $s(e) \neq s(f)$.
\end{definition}

\begin{theorem}\label{thm: TCH-thm}
Let $D$ be a multidigraph that does not contain any cycle-piercing edge. Then $\tch$ is W-chordal. As a result, the directed tree complex $\DT(D)$ is shellable.
\end{theorem}
\begin{proof}
We begin by showing that every edge of $D$ is simplicial.  
Afterward, we verify that both the deletion and the link-graph contraction of $D$
with respect to any edge $e \in E(D)$ yield multidigraphs that still
satisfy the same hypothesis as $D$.

Fix an edge $e \in E(D)$. First, assume there exists an edge
$f \in E(D)$ with $t(e)=t(f)$ and $s(e)\neq s(f)$. By the hypothesis, no directed cycle $D[\sigma]$ can contain $e$; otherwise, $f$ would be a cycle-piercing edge. Thus, the configuration required by the definition of a non-simplicial edge cannot arise, so $e$ must be a simplicial edge of $D$.

Now assume instead that $e$ lies on some directed cycle $D[\sigma]$.
Since $D$ contains no cycle-piercing edge, there is no edge $f \in E(D)$ with $t(e)=t(f)$ and $s(f)\neq s(e)$. Once again, the condition characterizing a non-simplicial edge cannot hold, so $e$ is simplicial. Hence, every edge of $D$ is simplicial.

Next, consider the multidigraph $D-e$. Any directed cycle in $D-e$
is also a directed cycle in $D$, since only the edge $e$ has been
removed. Therefore, if $D-e[\sigma]$ is a directed cycle and
$f \in E(D-e)$ satisfies $t(g)=t(f)$ and $s(f)\neq s(g)$ for some
edge $g\in \sigma$, then this same pair of edges already appears in $D$,
creating the same forbidden configuration in a directed cycle of $D$.
This contradicts the hypothesis on $D$. Hence, $D-e$ also satisfies
the same condition as $D$.

Finally, consider the multidigraph $D_{\downarrow e}$. By definition,
$D_{\downarrow e}$ is obtained from $D$ by deleting all edges that are
parallel to $e$, form a $2$-cycle with $e$, or have the same target as $e$,
and then identifying the vertices $s(e)$ and $t(e)$. Let
$D_{\downarrow e}[\sigma]$ be a directed cycle, and suppose there exists
an edge $f$ in $D_{\downarrow e}$ such that $t(g)=t(f)$ and
$s(f)\neq s(g)$ for some $g\in\sigma$. Since the construction of
$D_{\downarrow e}$ does not introduce new edges; the same pair of
edges already occurs in $D$, yielding the same configuration in a
directed cycle of $D$. This again contradicts the hypothesis on $D$.
Therefore, $D_{\downarrow e}$ also satisfies the required condition.
\end{proof}
\begin{example}

\begin{figure}[H]
\centering
\tikzset{vtx/.style={circle, fill=black, inner sep=1.5pt}}
    \begin{tikzpicture}
    \node[vtx] (a) at (0,0) {};
    \node[vtx] (c) at (2,-2) {};
    \node[vtx] (b) at (2,0) {};
    \node[vtx] (d) at (-2,0) {};
    \node[vtx] (e) at (4,0) {};
    \node[vtx] (f) at (4,-2) {};
    \node[vtx] (g) at (-4,0) {};
    \node[vtx] (h) at (-4,-2) {};

    \draw[->] (1.8,0)--(0.2,0) node[midway,above] {$e_1$};
    \draw[->] (-1.8,0)--(-0.2,0) node[midway,above] {$e_2$};
    \draw[->] (-4,-1.8) -- (-4,-0.2) node[midway,left] {$e_5$};
    \draw[->] (-3.8,0)--(-2.2,0) node[midway,above] {$e_3$};
    \draw[->] (-3.8,0.1)--(-2.2,0.1) node[midway,below] {$e_4$};
    \draw[->] (-2.2,-0.1)--(-3.8,-2) node[midway,above] {$e_6$};
    \draw[->] (3.8,0)--(2.2,0) node[midway,above] {$e_7$};
    \draw[->] (2,-0.2)--(2,-1.8) node[midway,left] {$e_8$};
    \draw[->] (2.2,-2) -- (3.8,-2) node[midway,below] {$e_9$};
    \draw[->] (2.2,-1.9) -- (3.8,-1.9) node[midway,above] {$e_{10}$};
    \draw[->] (4,-1.8) -- (4,-0.2) node[midway,right] {$e_{11}$};
    \end{tikzpicture}
    \caption{}\label{fig: TCH-thm-ex}
\end{figure}

Let $D$ be the multidigraph depicted in \Cref{fig: TCH-thm-ex}. Consider the directed cycle $D[\{e_3,e_5,e_6\}]$. For each edge $e \in \{e_5,e_6\}$, there is no $f \in E(D)$ with $t(e)=t(f)$. In contrast, for $e_3 \in \{e_3,e_5,e_6\}$, there exists some $f \in E(D)$ such that $t(e)=t(f)$ but $s(e)=s(f)$. Consequently, the directed cycle $D[\{e_3,e_5,e_6\}]$ satisfies the conditions of \Cref{thm: TCH-thm}. The other directed cycle in \Cref{fig: TCH-thm-ex} also satisfies the assumptions of \Cref{thm: TCH-thm}. Thus, $\DT(D)$ is shellable.
\end{example}

\subsection{Conflict hypergraph for the directed linear forest}

We now carry out an analogous construction for directed linear
forests. While the philosophy is identical, additional pairwise
conflicts must be incorporated to enforce both in-degree and
out-degree restrictions characteristic of directed paths.

\begin{definition}\label{def: conflict hypergraph for Dlf}
Let $D$ be a multidigraph. We associate to $D$ a hypergraph $\lfch$, called the \emph{linear-forest conflict hypergraph}, whose vertex set is $E(D)$, that is, $V(\lfch)=E(D)$. Its hyperedge set is the following set:

\[
\begin{aligned}
E(\lfch) = &\{\sigma \subseteq E(D) : D[\sigma] \text{ is a directed cycle}\} \;\cup\; \bigl\{\{e,f\} : e \text{ and } f \text{ are parallel edges}\bigr\}\\
& \;\cup\; \bigl\{\{e,f\} : s(e) = s(f) \text{~and~} t(e)\neq t(f)\bigr\} \;\cup\; \bigl\{\{e,f\} : s(e) \neq s(f) \text{~and~} t(e) = t(f)\bigr\}.
\end{aligned}
\]
\end{definition}

\begin{lemma}\label{Ind-for-H-2}
Let $D$ be a multidigraph. Then we have $\Dlf(D) = \Ind(\lfch)$.
\end{lemma}

\begin{lemma}
For the linear-forest conflict hypergraph, the following identities hold:
\begin{enumerate}
    \item $\lfch \setminus \{e\} =  \H_{D-e}^{\mathrm{lf}}$
    \item $\lfch / e = \H_{D \lnkate e}^{\mathrm{lf}} \cup \big\{f : \{e,f\} \in E(\lfch) \big\}$.
\end{enumerate}
\end{lemma}

The above results show that the directed linear forest complex,
like the directed tree complex, can be realized as the independence
complex of an associated conflict hypergraph. Moreover, deletion
and contraction in $\lfch$ correspond to natural operations on the multidigraph. Thus, the hypergraph framework provides a
uniform way to study directed linear forests using the language of
independence complexes and $W$-chordality.

\section{Concluding Remarks}\label{lastsection}
The results of this paper suggest several directions for further investigation, both on the combinatorial structure of directed forest-type complexes and on the interaction between directed graphs, hypergraph chordality, and topological properties of independence complexes.

\subsection*{Conflict graph and hypergraph realizability}
Throughout this paper, topological properties of directed linear forest and directed tree complexes are derived from structural properties of their associated conflict graphs and hypergraphs. 
This naturally raises a representation-theoretic question: which graphs and hypergraphs can arise as conflict objects of directed graphs?

\begin{problem}
Characterize the simple graphs that arise as linear-forest conflict graphs, tree conflict graphs for some multidigraph. 
Analogously, characterize the hypergraphs that arise as conflict hypergraphs. 
Are there forbidden induced subgraphs that cannot occur in conflict (hyper)graphs?
\end{problem}

These questions are analogous to classical realizability problems for line graphs, comparability graphs, and circuit hypergraphs. 
A structural characterization—possibly in terms of forbidden induced subgraphs or minors—would clarify the scope of the conflict-based approach and could lead to intrinsic criteria for shellability or vertex decomposability independent of the underlying directed graph.
Here is a list of conditions that more or less directly follow from the results of \Cref{sec:DLF}. 
\begin{enumerate}
    \item Every linear-forest conflict graph admits a clique cover indexed by the vertices of the associated multidigraph. Moreover, in any linear-forest conflict graph, the neighborhood of every vertex can be covered by at most three cliques.
    \item No linear-forest conflict graph contains an induced claw $K_{1, 3}$ in which the three leaves correspond to edges with pairwise distinct sources and targets. This follows from the fact that a center edge can only conflict with three others if at least two conflicts come from the same source or target, forcing adjacency among leaves.
    \item Every complete bipartite graph $K_{m,n}$ arises as a linear-forest conflict graph. Construct a multidigraph such that $m$ edges sharing a common source and $n$ edges sharing a common target, with no overlaps otherwise.
\end{enumerate}

\subsection*{Characterizing minimal forbidden configurations}
Our vertex decomposability results for directed linear forest complexes and directed tree complexes are obtained via the absence of certain induced configurations, such as alternating closed trails and directed cycles. 
A natural next step is to obtain a complete characterization of minimal forbidden induced subgraphs for which these complexes fail to be vertex decomposable or shellable.
\begin{problem}
Classify all minimal multidigraphs $D$ such that $\Dlf$ (and respectively $\DT$) is not vertex decomposable, but becomes vertex decomposable upon deletion of any edge.
\end{problem}

\subsection*{Homotopy types and finer invariants}
Although vertex decomposability and shellability provide strong combinatorial control, they do not fully determine the homotopy type of the associated complexes. 
In many examples, directed forest complexes exhibit rich topological behavior beyond wedge-of-spheres phenomena.
\begin{problem}
    Describe the homotopy types or homology groups of $\Dlf$ and $\DT$ for broader classes of multidigraphs, particularly in the presence of directed cycles.
\end{problem}

\subsection*{Algorithmic and enumerative questions}
The conflict graph and hypergraph constructions suggest algorithmic approaches to detecting vertex decomposability or shellability.
\begin{problem}
    Develop efficient algorithms to test vertex decomposability or shellability of these complexes via their conflict graphs or hypergraphs, and study the computational complexity of these problems.
\end{problem}

\nocite{*}
\bibliographystyle{abbrv}
\bibliography{references} 

@article{ATCF05,
author = {Francisco, Christopher and Tuyl, Adam},
year = {2005},
month = {12},
pages = {},
title = {Sequentially {C}ohen-{M}acaulay Edge Ideals},
volume = {135},
journal = {Proceedings of the American Mathematical Society},
doi = {10.2307/20534833}
}

@article{Woodroofe09,
   title={Vertex decomposable graphs and obstructions to shellability},
   volume={137},
   ISSN={0002-9939},
   url={http://dx.doi.org/10.1090/S0002-9939-09-09981-X},
   DOI={10.1090/s0002-9939-09-09981-x},
   number={10},
   journal={Proceedings of the American Mathematical Society},
   publisher={American Mathematical Society (AMS)},
   author={Woodroofe, Russ},
   year={2009},
   month=oct, pages={3235–3235} }

@article{Dusko13,
author = {Jojic, Dusko},
year = {2013},
month = {01},
pages = {1551-1559},
title = {Shellability of complexes of directed trees},
volume = {27},
journal = {Filomat},
doi = {10.2298/FIL1308551J}
}

@article{Engstrom05,
title = {Complexes of directed trees and independence complexes},
journal = {Discrete Mathematics},
volume = {309},
number = {10},
pages = {3299-3309},
year = {2009},
issn = {0012-365X},
doi = {https://doi.org/10.1016/j.disc.2008.09.033},
url = {https://www.sciencedirect.com/science/article/pii/S0012365X08005797},
author = {Alexander Engström},
}

@article{Ducko12,
author = {Jojic, Dusko},
year = {2012},
month = {01},
pages = {241-},
title = {Complexes of directed trees of complete multipartite graphs},
volume = {100},
journal = {Publications de l Institut Mathematique},
doi = {10.2298/PIM1206043J}
}

@article{kozlov99,
  title={Complexes of directed trees},
  author={Kozlov, Dmitry N},
  journal={Journal of Combinatorial Theory, Series A},
  volume={88},
  number={1},
  pages={112--122},
  year={1999},
  publisher={Elsevier}
}

@article {anurag21,
author = {Singh, Anurag},
title = {Vertex decomposability of complexes associated to forests},
journal = {Transactions on Combinatorics},
volume = {11},
number = {1},
pages = {1-13},
year  = {2022},
publisher = {University of Isfahan},
issn = {2251-8657}, 
eissn = {2251-8665}, 
doi = {10.22108/toc.2021.127059.1809}
}

@article{marietti08,
  title={A uniform approach to complexes arising from forests},
  author={Marietti, Mario and Testa, Damiano},
  journal={the electronic journal of combinatorics},
  pages={R101--R101},
  year={2008}
}

@article{JJ08,
author = {Jonsson, Jakob},
year = {2008},
month = {01},
pages = {},
title = {Simplicial Complexes of Graphs},
volume = {1928},
isbn = {978-3-540-75858-7},
journal = {Lecture Notes in Mathematics},
doi = {10.1007/978-3-540-75859-4}
}

@article{VANTUYL08,
title = {Shellable graphs and sequentially Cohen–Macaulay bipartite graphs},
journal = {Journal of Combinatorial Theory, Series A},
volume = {115},
number = {5},
pages = {799-814},
year = {2008},
issn = {0097-3165},
doi = {https://doi.org/10.1016/j.jcta.2007.11.001},
author = {Adam {Van Tuyl} and Rafael H. Villarreal}
}

@article{EHRENBORG06,
title = {The topology of the independence complex},
journal = {European Journal of Combinatorics},
volume = {27},
number = {6},
pages = {906-923},
year = {2006},
issn = {0195-6698},
doi = {https://doi.org/10.1016/j.ejc.2005.04.010},
author = {Richard Ehrenborg and Gábor Hetyei},
}

@book{kozlov2007,
  title={Combinatorial Algebraic Topology},
  author={Kozlov, D.},
  isbn={9783540719625},
  lccn={2007933072},
  series={Algorithms and Computation in Mathematics},
  url={https://books.google.co.in/books?id=BfBHAAAAQBAJ},
  year={2007},
  publisher={Springer Berlin Heidelberg}
}

@article {ProBill1980,
    AUTHOR = {Provan, J. Scott and Billera, Louis J.},
     TITLE = {Decompositions of simplicial complexes related to diameters of
              convex polyhedra},
   JOURNAL = {Math. Oper. Res.},
  FJOURNAL = {Mathematics of Operations Research},
    VOLUME = {5},
      YEAR = {1980},
    NUMBER = {4},
     PAGES = {576--594},
      ISSN = {0364-765X,1526-5471},
   MRCLASS = {52A25 (90C05)},
  MRNUMBER = {593648},
MRREVIEWER = {J.\ Parida},
       DOI = {10.1287/moor.5.4.576},
       URL = {https://doi.org/10.1287/moor.5.4.576},
}

@article{CaputiCollariDiTrani2022,
  author    = {Luca Caputi and Carlo Collari and Sabino Di Trani},
  title     = {Combinatorial and topological aspects of path posets, and multipath cohomology},
  journal   = {Journal of Algebraic Combinatorics},
  year      = {2022},
  volume    = {56},
  doi       = {10.1007/s10801-022-011xx-x},
}

@article{CaputiCollariDiTrani2024,
  author  = {Luca Caputi and Carlo Collari and Sabino Di Trani},
  title   = {Multipath cohomology of directed graphs},
  journal = {Algebraic \& Geometric Topology},
  year    = {2024},
  volume  = {24},
}

@misc{CaputiCollariTraniSmith2022,
      title={On the homotopy type of multipath complexes}, 
      author={Luigi Caputi and Carlo Collari and Sabino Di Trani and Jason P. Smith},
      year={arXiv:2208.04656},
      eprint={2208.04656},
      archivePrefix={arXiv},
      primaryClass={math.CO},
      journal={},
      url={https://arxiv.org/abs/2208.04656}, 
}

@misc{CaputiCollariSmith2026,
      title={Multipath complexes of bidirectional polygonal digraphs}, 
      author={Luigi Caputi and Carlo Collari and Jason P. Smith},
      year={arXiv:2601.05670},
      eprint={2601.05670},
      archivePrefix={arXiv},
      primaryClass={math.CO},
      journal={},
      url={https://arxiv.org/abs/2601.05670}, 
}

@article{RW11,
author = {Woodroofe, Russ},
year = {2011},
month = {10},
title = {Chordal and Sequentially Cohen-Macaulay Clutters},
volume = {18},
journal = {The Electronic Journal of Combinatorics},
doi = {10.37236/695}
}
	
\end{document}